\documentclass[oneside]{amsart}
\usepackage[justification=centering]{caption}
\usepackage[a4paper]{geometry}
\usepackage{amsmath,amsthm,amssymb,array,enumerate,parskip,graphicx,etoolbox,tabularx}
\usepackage[all]{xy}
\usepackage{mathtools}
\usepackage{xcolor} 
\usepackage[pagebackref]{hyperref}
\definecolor{dark-red}{rgb}{0.5,0.15,0.15}
\definecolor{dark-blue}{rgb}{0.15,0.15,0.6}
\definecolor{dark-green}{rgb}{0.15,0.6,0.15}
\hypersetup{
    colorlinks, linkcolor=dark-red,
    citecolor=dark-blue, urlcolor=dark-green
}

\usepackage[nameinlink,capitalise,noabbrev]{cleveref}
\usepackage{tikz}

\DeclareRobustCommand{\circledbullet}{
\tikz{
\pgfsetbaselinepointlater{\pgfpointanchor{X}{base}}
\pgfcircle{\pgfpointorigin}{0.12cm}
\pgfusepath{stroke}
\node (X) {$\bullet$};
}}
\numberwithin{equation}{section}

\DeclareMathOperator{\fib}{fib}
\DeclareMathOperator{\Ho}{Ho}

\DeclareMathOperator{\colim}{colim}
\DeclareMathOperator{\Fun}{Fun}
\DeclareMathOperator{\Sp}{Sp}
\DeclareMathOperator{\cofib}{cofib}

\newcount\tmp

\newcommand{\Z}{\mathbb{Z}}
\newcommand{\KGL}{\mathbf{KGL}}
\newcommand{\KQ}{\mathbf{KQ}}
\newcommand{\KO}{\mathbf{KO}}
\newcommand{\KU}{\mathbf{KU}}
\newcommand{\ku}{\mathbf{ku}}
\newcommand{\ko}{\mathbf{ko}}
\newcommand{\K}{\mathbf{K}}

\newcommand{\R}{\mathbf{R}}
\newcommand{\A}{\mathbf{A}}
\newcommand{\B}{\mathbf{B}}
\DeclareMathOperator{\Cat}{Cat}
\renewcommand{\L}{\mathbf{L}}
\newcommand{\C}{\mathbb{C}}

\renewcommand{\S}{\mathbf{S}}
\newcommand{\xr}{\xrightarrow}

\newcommand{\cal}{\mathcal}

\newcommand{\pics}{\mathfrak{pic}}
\usepackage{comment}

\makeatletter
\newcommand*{\myfnsymbolsingle}[1]{%
  \ensuremath{%
    \ifcase#1
    \or 
      *%
    \or 
      \dagger
    \or 
      \ddagger
    \or 
      \mathsection
    \or 
      \mathparagraph
    \else 
      \@ctrerr
    \fi
  }%
}
\makeatother

\renewcommand*{\backref}[1]{}
\renewcommand*{\backrefalt}[4]{%
  \ifcase #1 %
No citations.
  \or
(cit. on p. #2).%
  \else
(cit on pp. #2).%
  \fi%
}

\usepackage{alphalph}
\newalphalph{\myfnsymbolmult}[mult]{\myfnsymbolsingle}{}

\DeclareMathOperator{\Spec}{Spec}

\DeclareMathOperator{\pic}{Pic}

\DeclareMathOperator{\cell}{Cell}

\DeclareMathOperator{\Hom}{Hom}

\DeclareMathOperator{\Mod}{Mod}
\DeclareMathOperator{\End}{End}

\DeclareMathOperator{\Char}{char}

\DeclareMathOperator{\Thick}{Thick}

\DeclareMathOperator{\map}{map}

\DeclareMathOperator{\vcd}{vcd}
\DeclareMathOperator{\Tot}{Tot}

\newcommand{\Pic}{\operatorname{\mathcal{P}ic}}

\newcommand{\iHom}{\underline{\Hom}}

\newtheorem{thm}{Theorem}[section]

\newtheorem{lem}[thm]{Lemma}
\newtheorem{cor}[thm]{Corollary}

\newtheorem{prop}[thm]{Proposition} \theoremstyle{definition}
\newtheorem{rem}[thm]{Remark}
\newtheorem{defn}[thm]{Definition}

\newcommand{\cE}{\cal{C}\mathrm{ell}}

\Crefname{thm}{Theorem}{Theorems}
\Crefname{cor}{Corollary}{Corollaries}
\Crefname{conjecture}{Conjecture}{Conjectures}
\Crefname{rem}{Remark}{Remarks}
\Crefname{prop}{Proposition}{Propositions}
\Crefname{question}{Question}{Questions}
\Crefname{figure}{Figure}{Figures}
\Crefname{condition}{Condition}{Conditions}
\Crefname{lem}{Lemma}{Lemmas}

\newcommand*{\longhookrightarrow}{\ensuremath{\lhook\joinrel\relbar\joinrel\rightarrow}}
\title[The homotopy limit problem and the Picard group of Hermitian $K$-theory]{The homotopy limit problem and the cellular Picard group of Hermitian $K$-theory}
\address{Department of Mathematical Sciences, Norwegian University of Science and Technology, Trondheim}
\email{drew.k.heard@ntnu.no}
\newcommand{\cC}{\mathcal{C}}

\newcommand{\SH}{\mathcal{SH}}
\author{Drew Heard}

\begin{document}
\begin{abstract}
We use descent theoretic methods to solve the homotopy limit problem for Hermitian $K$-theory over quasi-compact and quasi-separated base schemes. As another application of these descent theoretic methods, we compute the cellular Picard group of 2-complete Hermitian $K$-theory over $\Spec(\C)$, showing that the only invertible cellular spectra are suspensions of the tensor unit.
\end{abstract}
\maketitle
\setcounter{tocdepth}{1}
\tableofcontents

\section{Introduction}
Let $\KGL$ denote the motivic spectrum representing algebraic $K$-theory, along with its $C_2$-action, and $\KQ$ the motivic spectrum representing Hermitian $K$-theory (both always taken over a fixed based scheme $S$). There is a map $f \colon \KQ \to \KGL$ which is the motivic analog of the map $f' \colon \KO \to \KU$ from real $K$-theory to topological $K$-theory in stable homotopy. The purpose of this paper is to investigate the ways in which $f$ behaves like $f'$. For example, there is an equivalence $\KO \simeq \KU^{hC_2}$, and the homotopy limit problem in motivic homotopy theory asks if there is an equivalence $\KQ \simeq \KGL^{hC_2}$. For a field $F$, let $\vcd_2(F)$ denote the mod-2 cohomological dimension of the absolute Galois group of $F(\sqrt{-1})$. R{\"o}ndigs--Spitzweck--{\O}stv\ae{}r \cite{rso_homlim} have shown that if $F$ is a field of $\Char(F) \ne 2$, and $\vcd_2(F)<\infty$, then the homotopy limit problem holds after $\eta$-completion over $S = \Spec(F)$, that is, there is an equivalence $\KQ^\wedge_{\eta} \simeq \KGL^{hC_2}$.

The map $f'$ is a faithful Galois extension \cite{rognes_galois}, and so there is a good theory of descent; there is an equivalence of $\infty$-categories $\Mod(\KO) \simeq \Mod(\KU)^{hC_2}$ \cite[Thm.~3.3.1]{mathew_stojanoska}. Depending on the choice of base scheme, the map $f$ may be a motivic Galois extension in the sense of \cite{mot_galois}, but even then it is only faithful when restricted to the full subcategory $\Mod(\KQ)^\wedge_{\eta}$ of $\eta$-complete $\KQ$-modules, see \Cref{defn:etacmpl}. Because of this, one cannot expect an equivalence of module categories as above.  Instead, we prove the following.

\newtheorem*{thm:symmonequiv}{\Cref{thm:symmonequiv}}
\begin{thm:symmonequiv}
   Let $S$ be a quasi-compact and quasi-separated base scheme, then there is a symmetric monoidal equivalence of $\infty$-categories
  \[
  \xymatrix{
\Mod(\KQ)_{\eta-\mathrm{cmpl}} \ar@<0.5ex>[r]^-{\sim} & \ar@<0.5ex>[l] \Mod(\KGL)^{hC_2}.
}
  \]
\end{thm:symmonequiv}
Here the left adjoint is given by extension of scalars, and the right adjoint is given by taking homotopy fixed points. As a corollary, we obtain the following solution of the homotopy limit problem for Hermitian $K$-theory, strengthening \cite[Thm.~1.2]{rso_homlim}. Here $\SH(S)$ denotes the stable motivic category over the base scheme $S$.

\newtheorem*{cor:monequivalence}{\Cref{cor:monequivalence}}
\begin{cor:monequivalence}
 There is an equivalence $\KQ^{\wedge}_{\eta} \simeq \KGL^{hC_2}$ in $\SH(S)$.
\end{cor:monequivalence}

The equivalence $\Mod(\KO) \simeq \Mod(\KU)^{hC_2}$ has been used by Mathew--Stojanoska and Gepner--Lawson to give a proof that the Picard group of invertible modules $\pic(\KO) \cong \Z/8$, generated by the suspension $\Sigma^1 \KO$. The situation becomes slightly more complicated in the motivic world. Existing techniques appear to be best suited for computing the invertible objects in the cellular subcategory, cf.~\Cref{lem:picalg} (note that classically every object in $\Mod(\R)$, for $\R$ a commutative ring spectrum, is cellular). Even then, in order to aid computability, we need to work not with $\Mod(\KQ)$, but with $\Mod(\KQ^\wedge_2)$, the 2-completion, and over $\Spec \C$, where it is known that $\KQ^\wedge_2 \simeq (\KGL^\wedge_2)^{hC_2}$. Here it is possible to completely describe $\pi_{\ast,\ast}(\KQ^\wedge_2)$ thanks to the computations of Isaksen and Shkembi \cite{isa_shk}. The similarity with the classical computation of $\pi_*(\KO)$ leads one to wonder if the Picard group $\pic_{\cE}(\KQ^\wedge_2)$ of invertible cellular $\KQ^\wedge_2$-modules is also just given by suspensions of the unit. Since $\KQ^\wedge_2$ is an $(8,4)$-periodic spectrum, this would give a Picard group of $\Z \oplus \Z/4$, and this turns out to be the case.

\newtheorem*{thm:piccalc}{\Cref{thm:piccalc}}
\begin{thm:piccalc}
Over $\Spec(\C)$ the cellular Picard group $\pic_{\cE}(\KQ^\wedge_2) \cong \Z \oplus \Z/4$.
\end{thm:piccalc}
The proof uses the methods introduced in \cite{mathew_stojanoska} and \cite{gepner_lawson}. Namely, for any symmetric monoidal $\infty$-category $\cC$, there exists a connective spectrum $\pics(\cC)$ with the property that $\pi_0(\pics(\cC)) \cong \pic(\cC)$, the Picard group of invertible objects in $\cC$. Using this, we construct a homotopy fixed point spectral sequence
\[
H^s(C_2,\pi_t\pics_{\cE}(\KGL^\wedge_2)) \implies \pi_{t-s}\pics_{\cE}(\KGL^\wedge_2)^{hC_2},
\]
whose abutment for $t-s=0$ has $\pic_{\cE}(\KQ^\wedge_2)$ as a subgroup. Using the observation that this Picard group must be at least $\Z \oplus \Z/4$, we use this spectral sequence to show that this is actually the whole Picard group.
\begin{rem}
    We consider the category $\Mod_{\cE}(\KQ^\wedge_2)$ with the tensor product $- \otimes_{\KQ^\wedge_2}-$. There are other possibilities; for example, one could take the 2-complete tensor product. In stable homotopy, there is an equivalence between $\pic(\KO^\wedge_2)$ and $\pic_{K(1)}(\KO^\wedge_2)$, the latter being the Picard group of $K(1)$-locally invertible $\KO^\wedge_2$-modules, where $K(1)$ is the first Morava $K$-theory. We are unable to even compute the 2-complete Picard group of $\KGL^\wedge_2$ in this case, however.
\end{rem}
\subsection*{Conventions}
We freely use the language of stable $\infty$-categories, as described by Lurie in \cite{ha}. Homotopy limits and colimits in an $\infty$-category will be denoted by $\varprojlim$ and $\varinjlim$ respectively. We use the notation $\iHom$ to denote the internal Hom in a closed symmetric monoidal $\infty$-category.

All schemes are assumed to be quasi-compact and quasi-separated.
\subsection*{Acknowledgments}
We thank Tobias Barthel, Marc Hoyois, Akhil Mathew, Niko Naumann, Paul-Arne {\O}stv\ae{}r, Birgit Richter, Markus Spitzweck, and Gabriel Valenzuela for helpful conversations regarding this work. We are particularly grateful to Markus Land for detailed comments on an earlier version of this document, and to the referee for helpful comments, including the extension to arbitrary base schemes. We thanks Hamburg Universit{\"a}t for its hospitality, and DFG Schwerpunktprogramm SPP 1786 for financial support. Finally, revisions of this work were done while the author was supported by a grant from the Trond Mohn foundation.

\section{Background}
In this section we provide the necessary background material on motivic homotopy, in particular the motivic spectra $\KGL$ and $\KQ$ representing algebraic $K$-theory and Hermitian $K$-theory respectively. We also introduce the category of cellular modules, as first studied by Dugger and Isaksen \cite{dugger_isaksen_cell}.
\subsection{Motivic homotopy theory}\label{sec:backmotivic}
Let $S$ be a base scheme. We will work in the Morel--Voevodsky motivic stable homotopy category $\SH(S)$ of $\mathbb{P}^1$-spectra over $S$ \cite{morel_voev}. In fact, Morel and Voevodsky work with Noetherian schemes of finite Krull dimension, however using \cite[App.~C]{Hoyois2014quadratic} $\SH(S)$ can be defined over an arbitrary base scheme. For a more detailed introduction to this category, we refer the reader to, for example \cite{voe_intro} or \cite{lev_intro}.  We  do, however, recall here that we have two notions of sphere in $\SH(S)$: the suspension spectrum of the simplicial sphere $S^1 \coloneqq \Delta^1/\partial \Delta^1$, and the suspension spectrum of the Tate sphere $\mathbb{G}_m \coloneqq \mathbb{A}^1 \setminus \{ 0 \}$, pointed by 1. As is customary we set $S^{a+b,b} \coloneqq S^a \otimes \mathbb{G}_m^{\otimes b}$. Under these conventions we have $\mathbb{P}^1 \simeq S^{2,1}$. Note that the stable motivic homotopy category is triangulated, with suspension $S^{1,0}$.

 The work of Morel--Voevodsky uses the techniques of model categories, however it is known that to each model category there exists an underlying $\infty$-category. To be very specific, Robalo \cite[Cor.~1.2]{robali_motivic} has shown that the $\infty$-category underlying $\SH(S)$, which we will also denote by $\SH(S)$, is a stable, presentably symmetric monoidal $\infty$-category.\footnote{ That is, a symmetric monoidal $\infty$-category which is presentable and whose associated tensor bifunctor preserves colimits separately in each variable.} It follows that for any commutative algebra object $\A$ in $\SH(S)$ - that is, a commutative motivic ring spectrum, there is a well-defined presentably symmetric monoidal stable $\infty$-category $\Mod(\A)$ of $\A$-modules \cite[Sec.~4.5]{ha}.

We recall also the non-nilpotent stable motivic Hopf map $\eta \colon S^{1,1} \to S^{0,0}$, induced by the projection $\mathbb{A}^2 \setminus \{0 \} \to \mathbb{P}^1$. Let $C(\eta)$ denote the cofiber of this map, and $C(\eta^k)$ for the cofiber of the iterations of $\eta$.
\begin{defn}\label{defn:etacmpl}
  The $\eta$-completion $M^\wedge_{\eta}$ of a motivic spectrum $M$ is the Bousfield localization of $M$ at $C(\eta)$, or equivalently, the limit $M^{\wedge}_{\eta} \simeq \varprojlim (M \otimes C(\eta^k))$.
\end{defn}

The classical stable homotopy category is compactly generated by the sphere spectrum, but the situation in motivic homotopy theory is much more complicated. In fact, $\SH(S)$ is compactly generated, but there are infinitely many generators - a description of the generators can be found in, for example, \cite[Thm.~9.2]{dugger_isaksen_cell}, but this is not important for our work. We simply note that the tensor unit $\S = \Sigma^\infty S_+$ is always compact, see \cite[Prop.~C.12]{Hoyois2014quadratic}, which implies that all the dualizable objects in $\SH(S)$ are compact . One can ask when $\SH(S)$ is compactly generated by dualizable objects, but the problem seems to be subtle. It is true when $S = \Spec(k)$ for a field $k$ admitting resolutions of singularities (for example, if $\Char(k) = 0$), but need not be true in general, see \cite[Rem.~8.2]{nso_landweber}.

For a motivic commutative ring spectrum $\A$ the unit map $\S \to \A$ gives rise to a standard adjunction \[\xymatrix{\SH(S) \ar@<0.5ex>[r]& \ar@<0.5ex>[l] \Mod(\A)}\] given by extension and restriction of scalars. By \cite[Corollary 4.2.3.5]{ha} restriction of scalars preserves colimits, and so its left adjoint preserves compact objects. In particular, the tensor unit $\A$ is compact in $\Mod(\A)$, so that all dualizable objects in $\Mod(\A)$ are compact.

We will predominately be interested in two motivic spectra, which we now introduce.
\subsection*{Algebraic and hermitian $K$-theory}
The motivic analog of topological $K$-theory $\KU$ is the motivic algebraic $K$-theory spectrum $\KGL$. In fact, the construction is similar to that of $\KU$ in topology; here the role of the unitary groups is played by the group schemes $\mathrm{GL}_n$ and their classifying spaces $\mathrm{BGL}_n$. Indeed, if we define $\mathrm{BGL} = \colim_n \mathrm{BGL}_n$, then $\KGL$ is the $\mathbb{P}^1$-spectrum defined by the constituent spaces
\[
\KGL = (\Z \times \mathrm{BGL},\Z \times \mathrm{BGL},\ldots),
\]
where the structure maps are defined in \cite[p.~600]{MR1648048}. Each of the group schemes $\mathrm{GL}_n$ have an involution given by inverse-transpose, which leads to a $C_2$-action on $\mathrm{GL}$, and in turn on $\mathrm{BGL}$. This gives $\KGL$ itself a natural $C_2$-action. The spectrum $\KGL$ is $(2,1)$-periodic; there is a Bott periodicity map
\[
\beta \colon \Sigma^{2,1} \KGL \xr{\sim} \KGL.
\]
Morel and Voevodsky \cite{morel_voev} (over regular Noetherian schemes) and Cisinski \cite[Thm.~2.20]{Cisinski2013Descente} (over a Noetherian scheme of finite Krull dimension, however the proof holds more generally over arbitrary base schemes, see also \cite[Sec.~5]{hoyois_cdh})  have shown that $\KGL$ represents Weibel's homotopy invariant algebraic $K$-theory\footnote{If the scheme $S$ is regular, then this agrees with Quillen's algebraic $K$-theory.} in the sense that for any smooth scheme $X$ we have
\begin{equation}\label{eq:rep}
\pi_0\Hom_{\SH(S)}(S^{p,q} \otimes \Sigma^{\infty} X_+, \KGL) = \KGL^{p,q}(X) \simeq \K_{p-2q}(X).
\end{equation}
In particular for $X = S$, we have $\pi_{p,q}\KGL \cong \K_{p-2q}(S)$.

It is known that $\KGL$ is a motivic commutative ring spectrum. For example, this follows from the identification of $\KGL \simeq \Sigma^{\infty}_+ \mathbb{P}^\infty[1/\beta]$, see \cite{so_bott,gs_bott}. In fact, the commutative ring structure is even unique \cite{nso_uniqueness}. As noted previously, the existence implies that there is a stable symmetric monoidal $\infty$-category $\Mod(\KGL)$ of modules over $\KGL$. Moreover, $\KGL$ itself is compact as a $\KGL$-module, so that all dualizable objects are compact. We note that the $\infty$-category $\Mod(\KGL)$ acquires a $C_2$ action (given by a map $BC_2 \to \Cat_{\infty}$ into the $\infty$-category of small $\infty$-categories \cite[Ch.~3]{lurie_htt}), arising from the $C_2$-action on $\KGL$ mentioned previously. Informally, for $g \in C_2$ this action sends a $\KGL$-module $M$ with structure map $\KGL \otimes M \xr{f} M$ to the motivic spectrum $g \cdot M$ which is $M$ as a motivic spectrum and with structure map $\KGL \otimes M \xr{g \otimes \text{id}} \KGL \otimes M \xr{f} M$.
\begin{rem}
  To make this slightly more precise, we note that the sending a commutative algebra to its category of modules is contravariantly functorial in the commutative algebra, and has a left adjoint given by extension of scalars \cite[Corollary 4.2.3.2, Corollary 4.2.3.7, and Remark 4.2.3.8]{ha}. In particular, the $C_2$-action on $\KGL$ descends to a $C_2$-action on the module category $\Mod(\KGL)$.
\end{rem}

With the representability of algebraic $K$-theory in the stable motivic homotopy category, one was led to wonder if Karoubi's Hermitian $K$-theory $\K^h$ was also representable. This was proved by Hornbostel in \cite{Hornbostel}, at least over Noetherian regular base schemes with $1/2 \in S$; there is a motivic spectrum $\KQ$ with the property that for a regular scheme $X$ we have
\begin{equation}\label{eq:KHrep}
\K^h_p(X) \cong \pi_0\Hom_{\SH(S)}(S^{p,0} \otimes \Sigma^\infty X_+,\KQ),
\end{equation}
where the left-hand side denotes the $p$-th Hermitian $K$-group of $X$. We refer the reader to either Hornbostel's original paper \cite{Hornbostel}, or to \cite[Sec.~4]{ro_slices} for further details, as well as the identification of \eqref{eq:KHrep} for all $S^{p,q}$.  In \cite[Sec.~4]{ro_slices} one can also find a construction of the forgetful map $f \colon \KQ \to \KGL$ used repeatedly in this paper. The ring spectrum $\KQ$ is also a motivic commutative ring spectrum, and the map $f$ is a map of commutative algebras, see \cite[Sec.~3.2.5]{bachmann_hopkins} (where $\KQ$ is denoted $KO$). Finally, we note that the work of Bachmann and Hopkins \cite[Sec.~3.2]{bachmann_hopkins} gives constructions of $\KGL$ and $\KQ$ over arbitrary base schemes, and shows moreover that these motivic spectra are stable under base change \cite[3.2.9]{bachmann_hopkins}.

\subsection{Cellularity}
As noted, one of the pleasant properties of the ordinary stable homotopy category is that, given a commutative ring spectrum $\R$, the $\infty$-category $\Mod(\R)$ of $\R$-modules is compactly generated by $\R$ itself. This is very far from true in the motivic homotopy category; for example, as already noted $\SH(S)$ has infinitely many compact generators. Instead, following \cite{dugger_isaksen_cell}, we can consider the subcategory of cellular objects, in the sense of the following definition.
\begin{defn}\label{def:cell}
Let $\A$ be a commutative motivic ring spectrum in $\SH(S)$. The full subcategory of cellular objects $\Mod_{\cE}(\A)$ is defined to be the localizing subcategory of $\Mod(\A)$ generated by the tensor unit $\A$ and its suspensions $\Sigma^{p,q}\A$. That is, $\Mod_{\cE}(\A)$ is the smallest thick subcategory of $\Mod(\A)$ containing $\A$ and its bigraded (de)suspensions that is closed under arbitrary colimits.
\end{defn}

If $M$ and $N$ are in $\Mod_{\cE}(\A)$, then so is $M \otimes_{\A} N$, since $-\otimes_{\A}-$ commutes with colimits. This implies that $\Mod_{\cE}(\A)$ is a symmetric monoidal $\infty$-category, see \cite[Rem.~2.2.1.2]{ha}. It is also presentable by \cite[Cor.~1.4.4.2]{ha}

We let $\Mod^{\omega}_{\cE}(\A)$ denote the compact objects in $\Mod_{\cE}(\A)$. These have a convenient description, which can be proved in the same way as \cite[Lem.~4.3]{nso_landweber}.
\begin{lem}
 The subcategory of compact objects $\Mod^\omega_{\cE}(\A)$ is equivalent to the thick subcategory of $\Mod(\A)$ generated by the tensor unit $\A$ and its suspensions $\Sigma^{p,q}A$.
\end{lem}

We observe that the inclusion $\Mod_{\cE}(\A) \stackrel{\iota}{\longhookrightarrow} \Mod(\A)$ preserves colimits by construction, and hence has a right adjoint, which we denote $\cell$  \cite[Cor.~5.5.2.9]{lurie_htt}. Moreover, since the inclusion functor preserves compact objects, $\cell$ preserves coproducts, and hence, by \cite[Prop.~1.4.4.1]{ha}, all colimits. As the right adjoint of a symmetric monoidal functor, $\cell$ is lax symmetric monoidal \cite[Cor.~7.3.2.7]{ha}.

We note the following, which also appears in \cite[Rem.~5.19]{bhv_2}.
\begin{lem}\label{lem:cellular}
If $M \in \Mod(\A)$, and $N \in \Mod_{\cE}(\A)$, then there is an equivalence
\[
\cell(M \otimes_{\A} N) \simeq \cell(M) \otimes_{\A} N.
\]
\end{lem}
\begin{proof}
  As noted, $\cell$ is lax symmetric monoidal, so there is a map
  \[
\cell(M \otimes_{\A} X) \to \cell(M) \otimes_{\A} \cell(X).
  \]
The collection of $X$ for which this is an equivalence is a localizing subcategory (since $\cell$ commutes with colimits) containing $\A$ (and its bigraded suspensions), and hence it is an equivalence for all cellular $X$.
\end{proof}
On the other hand, there is no reason for the internal Hom in $\Mod(\A)$, which we denote by $\iHom_{\A}(M,N)$, to be cellular. Rather, we have the following result.
\begin{lem}\label{lem:cellinthom}
\sloppy
The category of cellular $\A$-modules has an internal Hom given by $\cell (\iHom_{\A}(M,N))$.
  \end{lem}
\begin{proof}
We have the following natural chain of equivalences
\[
\begin{split}
\Hom_{\Mod_{\cE}(\A)}(M \otimes_{\A} S,T) &\simeq \Hom_{\Mod(\A)}(M \otimes_{\A} S,T) \\
& \simeq \Hom_{\Mod(\A)}(M, \iHom_{\A}(S,T)) \\
& \simeq \Hom_{\Mod_{\cE}(\A)}(M,\cell \iHom_{\A}(S,T))
\end{split}
\]
for $M,S,T \in \Mod_{\cE}(A)$.
\end{proof}

\section{Some descent theory}
A faithful $G$-Galois extension of ring spectra $f \colon \A \to \B$ gives rise to an equivalence of symmetric monoidal $\infty$-categories $\Mod(A) \simeq \Mod(B)^{hG}$ \cite[Thm.~3.3.1]{mathew_stojanoska}. The purpose of this section is to study morphisms of motivic ring spectra that are not quite motivic Galois extensions, but for which there is still a variant of descent that holds.
\subsection{Descent theory}
Once again, we assume that $S$ is a Noetherian scheme of finite Krull dimension. Suppose we are given a morphism $f \colon \A \to \B$ of motivic commutative ring spectra $\SH(S)$, making $\B$ a commutative $\A$-algebra, where $\B$ has an action of a finite group $G$ acting via commutative $\A$-algebra maps. The example to keep in mind is, of course, $f \colon \KQ \to \KGL$, however we work in more generality for now in the hope that there may be further applications in the future, such as to a (hypothetical) motivic version of topological modular forms.

\begin{defn}\label{defn:unramified}
The morphism $f \colon \A \to \B$ is called unramified if the the map $\B \otimes_{\A} \B \to \prod_{g \in G} \B$, given informally by $(b_1,b_2) \mapsto \{ b_1g(b_2)\}_{g \in G}$, is an equivalence.
\end{defn}

We note that this is half of the condition that $f$ is a motivic Galois extension in the sense of \cite{mot_galois}, however we do not immediately ask that the natural map $\A \to \B^{hG}$ is a weak equivalence. We also do not require $\B$ to be a faithful $\A$-module, which in turn implies that we do not have a good notion of Galois descent; there need not be an equivalence $\Mod(\A) \simeq \Mod(\B)^{hG}$. Nonetheless, we shall see that there is a weaker notion of descent that does still hold whenever $\B$ is a dualizable $\A$-module.

The following definition is well-known and used, as is the connection with local homology and local cohomology; for example, see \cite[Thm.~3.3.5]{hps}. In the context of $\infty$-categories complete objects have been studied in \cite{bhv1} and \cite{mnn_nilpotence}.
\begin{defn}\label{defn:complete}
  We say that $M \in \Mod(\A)$ is $\B$-complete if, for any $N \in \Mod(\A)$ with $N \otimes_\A \B \simeq 0$, we have that $\Hom_{\A}(N,M)$ is contractible.
\end{defn}
\sloppy This defines a full subcategory $\Mod(\A)_{\B-\mathrm{cmpl}}$ of $\B$-complete $\A$-modules. The inclusion $\Mod(\A)_{\B-\mathrm{cmpl}} \subset \Mod(\A)$ has a left adjoint which is the Bousfield localization $\L^\A_\B$ at $\B$ (in the category of $\A$-modules), see \cite[p.1008]{mnn_nilpotence} for example. This implies that $\Mod(\A)_{\B-\mathrm{cmpl}}$ is a symmetric monoidal $\infty$-category \cite[Prop.~2.2.1.9]{ha}.

In order to state the next reult, we recall the basic theory of homotopy fixed points of a category. Thus, suppose that $\cC$ is a presentably stable symmetric monoidal category $\cC$ with an action of a finite group $G$. This corresponds to an element of the functor category $\Fun(BG,\Cat_{\infty})$, where $BG$ is considered as an $\infty$-groupoid. By definition, the homotopy fixed points $\cC^{hG}$ is the limit of the diagram $\lim_{BG}\cC$. In the usual way, we can identify $BG$ with the colimit of the diagram $S_{\bullet} \colon N(\Delta)^{\text{op}} \to \cal{S}$ (where $N(\Delta)$ denotes the nerve of the simplex category and $\cal{S}$ is the category of spaces) with $S_q \simeq G^q$. It follows that $\cC^{hG}$ can be identified with the limit of the cosimplicial object $C^{\bullet}(G,\cC)$ with $C^q(G,\cC) \simeq \prod_{G^q} \cC$ and with coface and codegeneracies are induced by those on $BG$ (this is simply the Bousfield--Kan formula for a homotopy limit).  See also the discussion after Remark 4.13 of \cite{DAGXI}. Finally, we note that by \cite[Lemma 6.5.3.7]{lurie_htt} it suffices to take this limit over the semi-cosimplicial diagram without the codegeneracy maps; very explicitly, $\cC^{hG}$ can therefore be computed as the totalization
\[
\cC^{hG} \simeq \Tot\left(\xymatrix{\cC \ar@<0.5ex>[r] \ar@<-0.5ex>[r]& \displaystyle \prod_G \cC \ar@<1ex>[r] \ar@<0.3333ex>[r]\ar@<-0.3333ex>[r]\ar@<-1ex>[r] & \displaystyle \prod_{G \times G} \cC \cdots} \right)
\]

With this in mind, we have the following.
\begin{thm}\label{thm:descent}
Suppose $S$ is a Noetherian scheme of finite Krull dimension. Let $f \colon \A \to \B$ be an unramified morphism in $\SH(S)$ with $\B$ a dualizable $\A$-module. Then, there is a symmetric monoidal equivalence of $\infty$-categories
  \[
\Mod(\A)_{\B-\mathrm{cmpl}} \simeq \Mod(\B)^{hG}.
  \]
\end{thm}
\begin{proof}
    We wish to apply \cite[Thm.~2.30]{mnn_nilpotence} (with $\cC = \Mod(\A)$); in \emph{loc.~cit.} the authors work under the hypothesis of \cite[Hyp.~2.26]{mnn_nilpotence}, however a careful reading of the proof shows that it is not necessary to assume that $\Mod(\A)$ is compactly generated by dualizable objects. Indeed, the key input is Lurie's version of the Barr--Beck theorem \cite[Cor.~4.7.6.3]{ha} for which no such assumption is required (see also Footnote 5 of \cite{behrens_shah}).

     Hence, applying \cite[Thm.~2.30]{mnn_nilpotence} gives an equivalence of symmetric monoidal $\infty$-categories
\[
\Mod(\A)_{\B-\mathrm{compl}} \simeq \Tot\left(\xymatrix{\Mod(\B) \ar@<0.5ex>[r] \ar@<-0.5ex>[r]&  \Mod(\B \otimes_{\A} \B) \ar@<1ex>[r] \ar@<0.3333ex>[r]\ar@<-0.3333ex>[r]\ar@<-1ex>[r] & \cdots}  \right),
\]
where we have used the symmetric monoidal equivalence of $\infty$-categories $\Mod_{\Mod(\A)}(\B) \simeq \Mod(\B)$, see \cite[Cor.~3.4.1.9]{ha}. By assumption, $\B \otimes_{\A} \B \simeq \prod_{G} \B$, and similar for the higher terms. We also observe that the functor that sends a commutative algebra object to its category of modules commutes with finite products (see \cite[Proposition 6.16]{DAGVII}, noting that every stable presentable $\infty$-category is linear over the sphere spectrum \cite[Example 6.3]{DAGVII}), and therefore $\Mod(\prod_G \B) \simeq \prod_G \Mod(\B)$, and similar for the higher terms. By the discussion before the theorem, the right hand side identifies with the totalization computing the homotopy fixed points for the action of $G$ on the $\infty$-category $\Mod(\B)$, and the result follows.
\end{proof}

For applications to the Picard group, it is easiest to work with the subcategory of cellular objects in $\Mod(\A)$. In order to ensure that $\B$ is in this category, we assume that $\B$ is in the thick subcategory generated by $\A$ and its bigraded suspensions (in the category $\Mod(\A)$), which also implies that $\B$ is a dualizable $\A$-module.

We can then define a symmetric monoidal category $\Mod_{\cE}(\A)_{\B-\mathrm{cmpl}}$ of cellularly $B$-complete $\A$-modules, exactly as in \Cref{defn:complete} (where we test against $N \in \Mod_{\cE}(\A)$).

  \begin{rem}\label{rem:commadjoints}
Because $\cell$ is right adjoint to the inclusion, it follows from the definitions that if $N \in \Mod(\A)_{\B-\mathrm{cmpl}}$, then  $\cell (N) \in \Mod_{\cell}(\A)_{\B-\mathrm{cmpl}}$, i.e., cellularization restricts to a functor $\Mod(\A)_{\B-\mathrm{cmpl}} \to \Mod_{\cell}(\A)_{\B-\mathrm{cmpl}}$, see also \cite[Lem.~4.3(1)]{behrens_shah}. Moreover, by \cite[Lem.~4.3(3)]{behrens_shah} the adjunction $\iota \colon \Mod_{\cell}(\A) \leftrightarrows \Mod(\A) \colon \cell$ induces an adjunction $\iota' \colon \Mod_{\cell}(\A)_{\B-\mathrm{cmpl}} \leftrightarrows  \Mod(\A)_{\B-\mathrm{cmpl}}\colon \cell$ where $\iota'$ is fully faithful, $\cell$ is the restriction of the cellularization functor to $\Mod(\A)_{\B-\mathrm{cmpl}}$, and $\iota' \simeq \L_{\A}^{\B} \circ \iota$. One then identifies $\Mod_{\cE}(\A)_{\B-\mathrm{cmpl}}$ with the essential image of $ \cell$ restricted to $\Mod(\A)_{\B-\mathrm{cmpl}}$. Finally, we note that if $M \in \Mod_{\cell}(\A)_{\B-\mathrm{cmpl}}$, then it need not be $\B$-complete as an $\A$-module.
\end{rem}

\begin{thm}\label{thm:celldescent}
 Suppose $S$ is a Noetherian scheme of finite Krull dimension. Let $f \colon \A \to \B$ be an unramified morphism in $\SH(S)$ and suppose $\B$ is in the thick subcategory generated by $\A$. Then, there is a symmetric monoidal equivalence of $\infty$-categories
  \[
\Mod_{\cE}(\A)_{\B-\mathrm{cmpl}} \simeq \Mod_{\cE}(\B)^{hG}.
  \]
\end{thm}
\begin{proof}
  This is similar to the previous theorem. As observed already our conditions imply that $\B \in \Mod^\omega_{\cE}(\A)$, so that we just need to show that $\B$ is dualizable in $\Mod_{\cE}(\A)$, i.e., for any $\mathbf{Z} \in \Mod_{\cE}(\A)$ the natural map
  \[
\cell(\iHom_\A(\B,\A)) \otimes_\A \mathbf{Z} \to \cell(\iHom_\A(\B,\mathbf{Z}))
  \]
  is an equivalence. But, using \Cref{lem:cellular} and the fact that $\B$ is dualizable in $\Mod(\A)$, we have
  \[
\begin{split}
  \cell(\iHom_\A(\B,\A)) \otimes_\A \mathbf{Z} & \simeq \cell(\iHom_\A(\B,\A) \otimes_\A \mathbf{Z})\\
  & \simeq \cell (\iHom_{\A}(\B,\mathbf{Z})),
\end{split}
  \]
  as required. Now we once again apply \cite[Thm.~2.30]{mnn_nilpotence}, and argue as in the previous theorem.
\end{proof}
\subsection{Application to Hermitian $K$-theory}

As mentioned previously, the main example to keep in mind for the previous section was the morphism $f \colon \KQ \to \KGL$. We now apply the results to this example. We begin with the following.
\begin{lem}\label{lem:fconds}
   Let $S$ be an arbitrary base scheme, then $f$ is an unramified morphism, and $\KGL$ is in the thick subcategory generated by $\KQ$.
\end{lem}
\begin{proof}
The key point for both of these claims is the existence of a cofiber sequence \cite[Thm.~3.4]{ro_slices} of $\KQ$-modules
\begin{equation}\label{eq:wood}
\Sigma^{1,1} \KQ \xr{\eta} \KQ \xr{f} \KGL \to \Sigma^{2,1} \KQ.
\end{equation}
This fiber sequence is only shown to exist for finite a finite dimensional regular and separated Noetherian base scheme, however since it exists over $\Spec(\mathbb{Z})$ it exists over an arbitrary base scheme (recall that both $\KQ$ and $\KGL$ are stable under base change). As an immediate consequence of this cofiber sequence we see that $\KGL \in \Thick_{\Mod(\KQ)}(\KQ)$.

The claim that $\KGL \otimes_{\KQ} \KGL \xr{\simeq} \prod_{g \in C_2} \KGL$ is \cite[Prop.~4.3.2]{mot_galois}; the proposition there is not over the base schemes that we specify, however the input for the proof is the fiber sequence \eqref{eq:wood}, and Thm.~3.4 and Eqs.~16 and 17 of \cite{ro_slices}, both which hold for arbitrary schemes, again by base change.
\end{proof}
Let $\Mod(\KQ)_{\eta-\mathrm{cmpl}}$ denote the full-subcategory of $\Mod(\KQ)$ consisting of those objects whose underlying spectrum is $\eta$-complete in the sense of \Cref{defn:etacmpl}.
\begin{lem}\label{lem:etacomp}
   There is a symmetric monoidal equivalence of $\infty$-categories
  \[
\Mod(\KQ)_{\eta-\mathrm{cmpl}} \simeq \Mod(\KQ)_{\KGL-\mathrm{cmpl}}.
  \]
\end{lem}
\begin{proof}
By \cite[Eq.~(2.22)]{mnn_nilpotence} for any $M \in \Mod(\KQ)$ there is an equivalence
\[
\L^{\KQ}_{\KGL}M \simeq \varprojlim (\cofib(I^{\otimes(k+1)} \to \KQ)\otimes_{\KQ} M),
\]
where $I = \fib(\KQ \xr{f} \KGL) \simeq \Sigma^{1,1}\KQ$ by \eqref{eq:wood}. Note that there is no ambiguity in our use of $\varprojlim$ here; the limit of a diagram in $\Mod(\KQ)$ is computed on the underlying diagram of motivic spectra, see \cite[Cor.~4.2.3.3]{ha}.

We have that $I^{\otimes(k+1)} \simeq \Sigma^{k+1,k+1} \KQ$, so that $\cofib(I^{\otimes(k+1)} \to \KQ) \simeq C(\eta^{k+1}) \otimes \KQ$. It follows that
\[
\L^{\KQ}_{\KGL}M \simeq \varprojlim (C(\eta^{k+1}) \otimes M) \simeq M^{\wedge}_\eta
\]
for any $M \in \Mod(\KQ)$. This gives an equivalence of $\infty$-categories $\Mod(\KQ)_{\eta-\mathrm{cmpl}} \simeq \Mod(\KQ)_{\KGL-\mathrm{cmpl}}$ which is easily seen to be symmetric monoidal.
   \end{proof}
   \begin{rem}
     There are several other ways to see this. As noted by the referee, we have that a $\KQ$-module $M$ is $\KGL$-complete if and only if $\Hom_{\KQ}(N,M) \simeq 0$ for all $\KGL$-local $\KQ$-modules. The category of $\KGL$-local $\KQ$-modules is generated by $\{ G[\eta^{-1}] \otimes \KQ \mid G \in \cal{G} \}$ where $\cal{G}$ denotes a set of generators for $\SH(S)$ (to see this, one can use the explicit formula for localization \cite[Con.~3.4]{mnn_nilpotence}), and it suffices to test only against these modules. By adjunction, we then see that $M$ is $\KGL$-complete if and only if $\Hom_{\SH(S)}(G[\eta^{-1}],\iota M)$ for each $G$ in $\cal{G}$, where $\iota M$ denotes the underlying motivic spectrum of the $\KQ$-module $M$. This is precisely the statement that the underlying motivic spectrum of $M$ is $\eta$-complete.

     For yet another way, we note that \eqref{eq:wood} gives an equivalence $\KGL \simeq \KQ \otimes C(\eta)$. We recall that $\eta$-completion is Bousfield localization at $C(\eta)$. We then have $L^{\KQ}_{\KGL} \simeq L_{\KQ \otimes C(\eta)}^{\KQ} \simeq L_{C(\eta)}$, where we have used the obvious motivic generalization of \cite[Lem.~4.3]{BakerJeanneret2002Brave}. Thus, the $\KGL$-completion of a $\KQ$-module agrees with the $\eta$-completion of the underlying motivic spectrum.
   \end{rem}
\begin{rem}
  Since every $\KGL$-module is automatically $\KGL$-complete \cite[Ex.~2.18]{mnn_nilpotence}, the previous lemma gives an alternative proof that any $\KGL$-module is $\eta$-complete, which is a special case of \cite[Lem.~2.1]{rso_homlim}.
\end{rem}
The following is then our first main result.
\begin{thm}\label{thm:symmonequiv}
  There is a symmetric monoidal equivalence of $\infty$-categories
  \[
\Mod(\KQ)_{\eta-\mathrm{cmpl}} \simeq \Mod(\KGL)^{hC_2}.
  \]
 \end{thm}
\begin{proof}
  By \Cref{lem:fconds} we can apply \Cref{thm:descent}.   The proof is finished by \Cref{lem:etacomp}.
\end{proof}
We then have the following strengthening of \cite[Thm.~1.2]{rso_homlim}.
\begin{cor}\label{cor:monequivalence}
    There is an equivalence $\KQ^{\wedge}_{\eta} \simeq \KGL^{hC_2}$ in $\SH(S)$.
\end{cor}
\begin{proof}
  Since the equivalence is symmetric monoidal, the tensor unit of $\Mod(\KGL)^{hC_2}$ is equivalent, under the functor $(-)^{hC_2}$, to $\KQ^{\wedge}_{\eta}$, hence the result.
\end{proof}

We have a similar result for the cellular subcategory, which will prove useful in the sequel. Let us denote by $\Mod_{\cE}(\KQ)_{\eta-\mathrm{cmpl}}$ the essential image of $\cell$ applied to $\Mod(\KQ)_{\eta-\mathrm{cmpl}} \subset \Mod(\KQ)$ (note that these objects need not be $\eta$-complete in the usual sense).
\begin{lem}
There is a symmetric monoidal equivalence of $\infty$-categories
 \[\Mod_{\cE}(\KQ)_{\eta-\mathrm{cmpl}} \simeq \Mod_{\cE}(\KQ)_{\KGL-\mathrm{cmpl}} .\]
\end{lem}
\begin{proof}
  In light of \Cref{rem:commadjoints} applying $\cell$ to the equivalence of \Cref{lem:etacomp} gives the desired result.
\end{proof}

Finally, we point out the following, which will be useful in the Picard computations below.
\begin{lem}\label{lem:morita}
For a regular Noetherian scheme $S$, there is a symmetric monoidal equivalence of $\infty$-categories
\[
\Mod_{\cE}(\KGL) \simeq \Mod({\K(S)})
\]
between cellular $\KGL$-modules, and modules over the (connective) algebraic $K$-theory spectrum $\K(S)$. In the 2-complete setting over $\Spec(\C)$ we have
\[
\Mod_{\cE}(\KGL^{\wedge}_2) \simeq \Mod({\ku^{\wedge}_2}),
\]
where $\ku^\wedge_2$ denotes the 2-complete connective topological $K$-theory spectrum.
\end{lem}
\begin{proof}
Recall that $\KGL$ is $(2,1)$-periodic; there is an equivalence $\Sigma^{2,1}\KGL \simeq \KGL$. Hence the localizing subcategory generated by $\{ \Sigma^{p,q}\KGL \}_{p,q \in \Z}$ is the same as the localizing subcategory generated by $\{ \Sigma^{p,0} \KGL \}_{p \in \Z}$. Since the suspension functor in $\SH(S)$ is given by smashing with $S^{1,0}$, the derived Morita theory of Lurie--Schwede--Shipley \cite[Prop.~7.1.2.7]{ha} gives an equivalence of symmetric monoidal $\infty$-categories
  \[
\Mod_{\cE}(\KGL) \simeq \Mod(\R),
  \]
  where $\R \simeq \End_{\Mod(\KGL)}(\KGL)$ naturally has the structure of a commutative ring spectrum. But, by \eqref{eq:rep} and adjunction, $\R \simeq \K(S)$. The argument in the 2-complete setting is similar; Morita theory tells us that $\Mod_{\cE}(\KGL^\wedge_2) \simeq \Mod(\K(\C)^{\wedge}_2)$. But Suslin's theorem \cite{suslin} implies that $\K(\C)^{\wedge}_2 \simeq \ku^{\wedge}_2$.
\end{proof}
\begin{rem}
  If $S$ is not regular Noetherian, then there is a similar result where Quillen's algebraic $K$-theory is replaced by Weibel's homotopy invariant $K$-theory (the two agree when $S$ is regular Noetherian).
\end{rem}
\begin{rem}
  At first this result may seem odd since $\KGL$ is $(2,1)$-periodic, while $\K(S)$ is connective, however note that
  \[
\pi_t\R \cong \pi_t\End_{\Mod(\KGL)}(\KGL) \cong \pi_0\Hom_{\SH(S)}(S^{t,0},\KGL) \cong \pi_{t,0}\KGL
  \]
  is indeed 0 for $t < 0$.
\end{rem}
\section{The cellular Picard group of Hermitian $K$-theory}
In this section we use the previous results on descent theory to give a calculation of the Picard group of invertible cellular $\KQ^\wedge_2$-modules over $\Spec(\C)$.
\subsection{Picard groups and Picard spectra}
We first recall the notation of the Picard spectrum from the work of Mathew and Stojanoska \cite{mathew_stojanoska}, referring the reader to either their paper, or to work of Gepner and Lawson \cite{gepner_lawson} for more information. Let $\cC$ be a symmetric monoidal presentable $\infty$-category with monoidal unit $\mathbf{1}$, and let $\pic(\cC)$ denote the group of isomorphism classes of invertible objects in $\Ho(\cC)$ (since $\cC$ is assumed to be presentable, the Picard group does indeed form a set, see \cite[Rem.~2.1.4]{mathew_stojanoska}). It turns out to be more useful to remember more information than just isomorphism classes of invertible objects; rather we should remember all higher equivalences between objects as well.  Thus, we let $\Pic(\cC)$ denote the $\infty$-groupoid of invertible objects in $\cC$ and equivalences between them. This is a grouplike $\mathbb{E}_\infty$-space, and hence there exists a connective spectrum $\pics(\cC)$ with $\Omega^{\infty}\pics(\cC) \simeq \Pic(\cC)$. We have the following description of the homotopy groups of $\pics(\cC)$:
\begin{equation}\label{eq:pics}
\pi_i\pics(\cC) =   \begin{cases}
\pic(\cC) & i = 0, \\
\pi_0(\End(\mathbf{1}))^\times & i=1, \text{ and } \\
\pi_{i-1}(\End(\mathbf{1})) & i \ge 2.
\end{cases}
\end{equation}
\begin{defn}
  For a commutative motivic ring spectrum $\A$ we let $\pics(\A)$ denote $\pics(\Mod(\A))$. For the symmetric monoidal $\infty$-category $\Mod(\A)_{\B-\mathrm{cmpl}}$, we let $\pics(\A)_{\B-\mathrm{cmpl}}$ denote $\pics(\Mod(\A)_{\B-\mathrm{cmpl}})$. We similarly define the Picard spaces $\Pic(\A)$ and $\Pic(\A)_{\B-\mathrm{cmpl}}$ as the Picard spaces of the respective monoidal categories.
\end{defn}

As a functor $\pics \colon \Cat^\infty \to \Sp_{\ge 0}$ from the $\infty$-category of symmetric monoidal $\infty$-categories to the $\infty$-category of connective spectra, $\pics$ has the important property that it commutes with limits \cite[Prop.~2.2.3]{mathew_stojanoska} (and similar for $\Pic$, as a functor to spaces). Thus, we get the following from \Cref{thm:descent}.
 \begin{prop}\label{prop:picdescent}
 Suppose $S$ is a Noetherian scheme of finite Krull dimension. Let $f \colon \A \to \B$ be an unramified morphism in $\SH(S)$ with $\B$ a dualizable $\A$-module. Then, there is an equivalence of connective spectra
  \[
\pics(\A)_{\B-\mathrm{cmpl}} \simeq \tau_{\ge 0} \pics(\B)^{hG}.
  \]
  There is an associated homotopy fixed point spectral sequence
  \[
H^s(G;\pi_t\pics(\B)) \implies \pi_{t-s} \pics(\B)^{hG}
  \]
whose abutment for $t = s$ is the Picard group $\pic(\A)_{\B-\mathrm{cmpl}}$.
\end{prop}
Of course, we are usually more interested in $\pic(\A)$ itself, however we note the following.
\begin{lem}\label{lem:picsubgroup}
  If $\A$ is $\B$-complete, then the Picard group $\pic(\A) \subseteq \pic(\A)_{\B-\mathrm{cmpl}}$.
\end{lem}

\begin{proof}
  It is a general fact about Picard groups that all elements of $\pic(\A)$ are dualizable in $\Mod(\Ho(\A))$, see \cite[Prop.~A.2.8]{hps}. This implies that for $M \in \pic(\A)$ we have $L_{\B}^{\A}M \simeq L_{\B}^{\A}\A \otimes_{\A} M$ by \cite[Lem.~3.3.1]{hps}. Since $\A$ is $\B$-complete, $L_{\B}^{\A}\A \simeq \A$, and it is easy to see that this implies the result.
\end{proof}

Unfortunately, it seems difficult to approach the Picard group directly. For example, Baker and Richter \cite{baker_richter} showed that for $\R$ a commutative ring spectrum, there is an injection $\Phi \colon \pic(\R_*) \to \pic(\R)$, where $\pic(\R_*)$ denotes the algebraic Picard group if invertible $\R_*$-modules. Although we will not use it, it is worthwhile to point out the following, where we let $\pic_{\cE}(\A) = \pic(\Ho(\Mod_{\cE}(\A))$.
\begin{lem}\label{lem:picalg}
  For a motivic commutative ring spectrum $\A$, there is an injection $\Phi \colon \pic(\A_{\ast,\ast}) \to \pic_{\cE}(\A)$.
\end{lem}
\begin{proof}
  The same proof as given by Baker--Richter \cite[Sec.~1]{baker_richter} works in this context. Namely, given an invertible module $\overline M \in \pi_{\ast,\ast}\A$ (which is of course projective), one can always produce an $\A$-module $M$ with $\pi_{\ast,\ast}M \cong \overline{M}$. Inspection of the proof reveals that $M$ is produced as a colimit of free $\A$-modules, and so is an element of the cellular Picard group $\pic_{\cE}(\A)$. In fact, to check that $\Phi$ gives  a well-defined group homomorphism relies on the existence of a K{\"u}nneth spectral sequence, which does exist for cellular objects, see \cite[Prop.~7.10]{dugger_isaksen_cell}. This is a monomorphism because $\pi_{\ast,\ast}$ detects weak equivalences between cellular motivic spectra, see \cite[Sec.~7.9]{dugger_isaksen_cell}.
\end{proof}
If we focus our attention the cellular Picard group, then we have the following analog of \Cref{prop:picdescent}, which follows from \Cref{thm:celldescent}.
 \begin{prop}\label{prop:piccelldescent}
 Suppose $S$ is a Noetherian scheme of finite Krull dimension. Let $f \colon \A \to \B$ be an ramified morphism in $\SH(S)$, and suppose $\B$ is in the thick subcategory generated by $\A$. Then, there is an equivalence of connective spectra
  \[
\pics_{\cE}(\A)_{\B-\mathrm{cmpl}} \simeq \tau_{\ge 0} \pics_{\cE}(\B)^{hG}.
  \]
  There is an associated homotopy fixed point spectral sequence
  \[
H^s(G;\pi_t\pics_{\cE}(\B)) \implies \pi_{t-s} \pics_{\cE}(\B)^{hG}
  \]
whose abutment for $t = s$ is the Picard group $\pic_{\cE}(\A)_{\B-\mathrm{cmpl}}$.
\end{prop}
The analog of \Cref{lem:picsubgroup} is the following.
\begin{lem}\label{lem:piccellsubgroup}
  If $\A$ is a cellularly $\B$-complete $\A$-module, then the Picard group $\pic_{\cE}(\A) \subseteq \pic_{\cE}(\A)_{\B-\mathrm{cmpl}}$.
\end{lem}

\subsection{The Picard group of 2-complete Hermitian $K$-theory}
In contrast to stable homotopy theory, where the homotopy groups of $\KU$ and $\KO$ are well-known, the situation for $\KQ$ and $\KGL$ is much more difficult, since even computing $\pi_{\ast,\ast}\KGL$ is equivalent to computing the homotopy invariant algebraic $K$-theory of the base scheme $S$. Thus, we restrict our attention to studying the 2-completions $\KQ^\wedge_2$ and $\KGL^\wedge_2$ over $\Spec(\C)$. From now on, for notational simplicity, \emph{all spectra are implicitly 2-completed}; that is we write $\KQ$ for $\KQ^\wedge_2$ and $\KGL$ for $\KGL^\wedge_2$. However, our monoidal product is \emph{not} 2-completed.

 Starting with Suslin's computations of the algebraic $K$-theory of fields \cite{suslin}, Isaksen and  Shkembi \cite{isa_shk}  compute that
\[
\pi_{\ast,\ast}\KGL \cong \Z_2[\tau,\beta^{\pm 1}],
\]
where $|\tau| = (0,-1)$ and $|\beta| = (2,1)$
and
\[
\pi_{\ast,\ast}\KGL^{hC_2} \cong \frac{\Z_2[\tau,h_1,a,b^{\pm 1}]}{2h_1,\tau h_1^3,a^2 - 4b,h_1a}
\]
where $|h_1| = (1,1)$, $|a| = (4,2)$ and $|b| = (8,4)$. Moreover, by \cite[Cor.~4.6]{kobal_k} (or \cite{hko_limit}) we have the following.
\begin{prop}\label{prop:kglfixepdoints}
  There is an equivalence $\KGL^{hC_2} \simeq \KQ$.
\end{prop}

In this section we use the methods of the previous section to compute $\pic_{\cE}(\KQ)$. The beginning is the following specialization of \Cref{prop:piccelldescent}.
 \begin{prop}\label{prop:piccelldescentkgl}
 There is an equivalence of connective spectra
  \[
\pics_{\cE}(\KQ)_{\KGL-\mathrm{cmpl}} \simeq \tau_{\ge 0} \pics_{\cE}(\KGL)^{hC_2},
  \]
  with associated homotopy fixed point spectral sequence
\begin{equation}\label{eq:picsskgl}
H^s(C_2,\pi_t\pics_{\cE}(\KGL)) \implies \pi_{t-s}\pic_{\cE}(\KGL)^{hC_2}.
\end{equation}
whose abutment for $t = s$ is the Picard group $\pic_{\cE}(\KQ)_{\KGL-\mathrm{cmpl}}$.
\end{prop}

To compute this spectral sequence, we start with a computation of $\pic_{\cE}(\KGL)$. We believe it is possible to use \Cref{lem:picalg} and a variant of \cite[Thm.~38]{baker_richter} to compute this, however it seems simpler to use \Cref{lem:morita}, which says that there is an equivalence of symmetric monoidal $\infty$-categories $\Mod_{\cE}(\KGL) \simeq \Mod(\ku)$.

We recall that $\KGL$ is $(2,1)$-periodic, that is, there is an equivalence $\Sigma^{2,1} \KGL \simeq \KGL$. Then, any suspension $\Sigma^{p,q}\KGL$ can be written in the form $\Sigma^{p-2q,0}\KGL$. This implies that there is at least a copy of the integers in the cellular Picard group of $\KGL$. We claim this is everything.
\begin{lem}\label{cor:algpic}
  The cellular Picard group $\pic_{\cE}(\KGL) \cong \Z$ generated by suspensions.
\end{lem}
\begin{proof}
  By \Cref{lem:morita} it is enough to compute $\pic(\ku)$, but by \cite[Thm.~21]{baker_richter} this Picard group is algebraic, i.e., $\pic(\ku) \cong \pic(\pi_*\ku)$, and is isomorphic to $\Z$.
\end{proof}
\begin{rem}
  We note that we have computed the Picard group with respect to the monoidal product $ - \otimes_{\KGL} -$, and not $(- \otimes_{\KGL} -)^\wedge_2$.
\end{rem}
We are now in a position to compute $\pic_{\cE}(\KQ)$. To do so, we will use the spectral sequence \eqref{eq:picsskgl}. Recall that $\KQ$ is an $(8,4)$-periodic spectrum, and hence (by considering various suspensions of the unit object), the cellular Picard group contains a copy of $\Z \oplus \Z/4$. This turns out to be everything.
\begin{thm}\label{thm:piccalc}
\begin{enumerate}
  \item The Picard group $\pic_{\cE}(\KQ)_{\KGL-\mathrm{cmpl}} \cong \Z \oplus \Z/4$ generated by suspensions of $\KQ$.
\item There is an isomorphism
\[
\pic_{\cE}(\KQ) \cong \pic_{\cE}(\KQ)_{\KGL-\mathrm{cmpl}} \cong \Z \oplus \Z/4.
\]
\end{enumerate}
\end{thm}
\begin{proof}
We claim that since $\KQ \simeq \KGL^{hC_2}$, and $\KGL$ is a cellularly $\KGL$-complete $\KQ$-module, $\KQ$ itself is a cellularly $\KGL$-complete $\KQ$-module. Indeed, it follows we have a series of equivalences
\[
\overline{\L}^{\KQ}_{\KGL}(\KQ) \simeq \cell \KQ^{\wedge}_{\eta} \simeq \cell (\KGL^{hC_2})^{\wedge}_{\eta} \simeq \cell \KGL^{hC_2} \simeq \cell \KQ \simeq \KQ.
\]

 It follows by \Cref{lem:piccellsubgroup} that $\Z \oplus \Z/4 \subseteq \pic_{\cE}(\KQ) \subseteq \pic_{\cE}(\KQ)_{\KGL-\mathrm{cmpl}}$, and we see that it suffices to prove the first statement of the theorem. We will use the descent spectral sequence for $\pics_{\cE}(\KQ)_{\KGL-\mathrm{cmpl}}$ to show that $\Z \oplus \Z/4$ is also an upper bound.

Thus, using \Cref{prop:piccelldescentkgl}, we want to compute the 0-stem of the Picard spectral sequence
\begin{equation}
E_2^{s,t} = H^s(C_2,\pi_t\pics_{\cE}(\KGL)) \implies \pi_{t-s}\pic_{\cE}(\KGL)^{hC_2}.
\end{equation}
  By \Cref{lem:morita}, \Cref{cor:algpic}, and \eqref{eq:pics} we have
\[
\pi_t\pics_{\cE}(\KGL) \cong
\begin{cases}
  \Z & t = 0 \\
  \Z_2^\times & t = 1 \\
  \pi_{t-1}\ku & t \ge 2.
\end{cases}
  \]
  Moreover, the equivalence  between $\Hom_{\Mod(\KGL)}(\KGL,\KGL) \simeq \Hom_{\SH(S)}(S,\KGL)$ and $\ku$ is $C_2$-equivariant for the usual action on $\ku$, see the proof of Theorem 3.18 of \cite{bachmann_hopkins}. Thus, we can completely determine the $E_2$-page of the spectral sequence in the range $t \ge 2$; we have additively
  \[
E_2^{s,t} = H^s(C_2;\pi_t \pics_{\cE}(\KGL)) \cong H^s(C_2;\pi_{t-1}\ku) \cong \begin{cases}
  \Z_2 & s = 0, t \equiv 1 \pmod 4 \\
  \Z/2 & s> 0, t \equiv 1 \pmod 4 \\
  0 & \text{else.}
\end{cases}
  \]

This almost completely determines the $t \ge 2$ range of the Picard spectral sequence. Arguing as in \cite[Cor.~5.2.3]{mathew_stojanoska} (and using an analog of the Comparison Tool 5.2.4 of \emph{loc.~cit.}) the differentials originating at $E_3^{s,t}$ for $t \ge 4$ in the Picard spectral sequence can be identified with those originating at $\overline{E}_3^{s,t-1}$ in the usual spectral sequence
\[
\overline{E}_2^{s,t} = H^s(C_2;\pi_t\ku) \implies \pi_{t-s}(\ko).
\]

We are now in the situation shown in \Cref{fig:hfpss}, which is drawn using Adams indexing. Note that we cannot import the differential originating at $(t-s,s) = (0,3)$ since this lies on the $t = 3$ line. One could prove and use an analog of \cite[Thm.~6.1.1]{mathew_stojanoska} to calculate this differential (which would imply that it is trivial), however we observe that we already have an upper bound of $\Z \oplus \Z/4$, which is also our lower bound. Moreover, the copy of $\Z$ in $E_2^{0,0}$, which corresponds to the various suspensions $\Sigma^{k,0}\KGL$ is represented by $\Sigma^{k,0}\KQ$ in $\pic_{\cE}(\KQ)_{\KGL-\mathrm{cmpl}}$. This also generates the copy of $\Z$ in the lower bound of $\Z \oplus \Z/4$ and so we conclude that $\pic_{\cE}(\KQ)_{\KGL-\mathrm{cmpl}} \cong \Z \oplus \Z/4$, as claimed.
\end{proof}
\begin{figure}[!tbhp]
\centering
\includegraphics{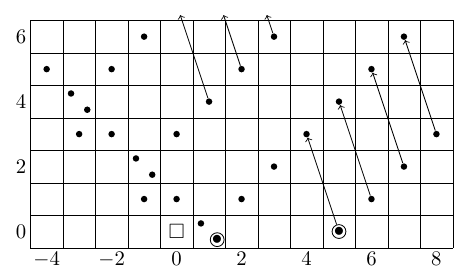}
\caption{The HFPSS for $\pi_*\pics_{\cE}(\KGL)^{hC_2}$. Here $\square$ denotes copies of $\Z$, $\bullet$ copies of $\Z/2$, and $\circledbullet$ copies of $\Z_2$. We have omitted some terms in negative filtration degree. }\label{fig:hfpss}\vspace{-4mm}
\end{figure}
\bibliography{motpic}

\newcommand{\etalchar}[1]{$^{#1}$}
\providecommand{\bysame}{\leavevmode\hbox to3em{\hrulefill}\thinspace}
\providecommand{\MR}{\relax\ifhmode\unskip\space\fi MR }
\providecommand{\MRhref}[2]{%
  \href{http://www.ams.org/mathscinet-getitem?mr=#1}{#2}
}
\providecommand{\href}[2]{#2}
\begin{thebibliography}{BHK{\etalchar{+}}18}

\bibitem[BH20]{bachmann_hopkins}
Tom Bachmann and Michael~J. Hopkins, \emph{{$\eta$}-periodic motivic stable
  homotopy theory over fields}, arXiv:2005.06778.

\bibitem[BHK{\etalchar{+}}18]{mot_galois}
Agn\`es Beaudry, Kathryn Hess, Magdalena K\c{e}dziorek, Mona Merling, and Vesna
  Stojanoska, \emph{Motivic homotopical {G}alois extensions}, Topology Appl.
  \textbf{235} (2018), 290--338. \MR{3760206}

\bibitem[BHV18a]{bhv_2}
Tobias Barthel, Drew Heard, and Gabriel Valenzuela, \emph{Local duality for
  structured ring spectra}, J. Pure Appl. Algebra \textbf{222} (2018), no.~2,
  433--463. \MR{3694463}

\bibitem[BHV18b]{bhv1}
\bysame, \emph{Local duality in algebra and topology}, Adv. Math. \textbf{335}
  (2018), 563--663. \MR{3836674}

\bibitem[BJ02]{BakerJeanneret2002Brave}
Andrew Baker and Alain Jeanneret, \emph{Brave new {H}opf algebroids and
  extensions of {$M$}{U}-algebras}, Homology Homotopy Appl. \textbf{4} (2002),
  no.~1, 163--173. \MR{1937961}

\bibitem[BR05]{baker_richter}
Andrew Baker and Birgit Richter, \emph{Invertible modules for commutative
  {$\Bbb S$}-algebras with residue fields}, Manuscripta Math. \textbf{118}
  (2005), no.~1, 99--119. \MR{2171294}

\bibitem[BS20]{behrens_shah}
Mark Behrens and Jay Shah, \emph{{$C_2$}-equivariant stable homotopy from real
  motivic stable homotopy}, arXiv:1908.08378, To appear in \emph{Annals of
  {$K$}-theory}.

\bibitem[Cis13]{Cisinski2013Descente}
Denis-Charles Cisinski, \emph{Descente par \'{e}clatements en {$K$}-th\'{e}orie
  invariante par homotopie}, Ann. of Math. (2) \textbf{177} (2013), no.~2,
  425--448. \MR{3010804}

\bibitem[DI05]{dugger_isaksen_cell}
Daniel Dugger and Daniel~C. Isaksen, \emph{Motivic cell structures}, Algebr.
  Geom. Topol. \textbf{5} (2005), 615--652. \MR{2153114}

\bibitem[DL{\O}{\etalchar{+}}07]{voe_intro}
B.~I. Dundas, M.~Levine, P.~A. {{\O}stv{\ae}r}, O.~R\"ondigs, and V.~Voevodsky,
  \emph{Motivic homotopy theory}, Universitext, Springer-Verlag, Berlin, 2007,
  Lectures from the Summer School held in Nordfjordeid, August 2002.
  \MR{2334212}

\bibitem[GL16]{gepner_lawson}
D.~{Gepner} and T.~{Lawson}, \emph{{Brauer groups and Galois cohomology of
  commutative ring spectra}}, ArXiv e-prints (2016).

\bibitem[GS09]{gs_bott}
David Gepner and Victor Snaith, \emph{On the motivic spectra representing
  algebraic cobordism and algebraic {$K$}-theory}, Doc. Math. \textbf{14}
  (2009), 359--396. \MR{2540697}

\bibitem[HKO11]{hko_limit}
P.~Hu, I.~Kriz, and K.~Ormsby, \emph{The homotopy limit problem for {H}ermitian
  {K}-theory, equivariant motivic homotopy theory and motivic {R}eal
  cobordism}, Adv. Math. \textbf{228} (2011), no.~1, 434--480. \MR{2822236}

\bibitem[Hor05]{Hornbostel}
Jens Hornbostel, \emph{{$A^1$}-representability of {H}ermitian {$K$}-theory and
  {W}itt groups}, Topology \textbf{44} (2005), no.~3, 661--687. \MR{2122220}

\bibitem[Hoy14]{Hoyois2014quadratic}
Marc Hoyois, \emph{A quadratic refinement of the
  {G}rothendieck-{L}efschetz-{V}erdier trace formula}, Algebr. Geom. Topol.
  \textbf{14} (2014), no.~6, 3603--3658. \MR{3302973}

\bibitem[{Hoy}20]{hoyois_cdh}
Marc {Hoyois}, \emph{{C}dh descent in equivariant homotopy {$K$}-theory}, {Doc.
  Math.} \textbf{25} (2020), 457--482.

\bibitem[HPS97]{hps}
Mark Hovey, John~H. Palmieri, and Neil~P. Strickland, \emph{Axiomatic stable
  homotopy theory}, Mem. Amer. Math. Soc. \textbf{128} (1997), no.~610, x+114.
  \MR{1388895}

\bibitem[IS11]{isa_shk}
Daniel~C. Isaksen and Armira Shkembi, \emph{Motivic connective {$K$}-theories
  and the cohomology of {A}(1)}, J. K-Theory \textbf{7} (2011), no.~3,
  619--661. \MR{2811718}

\bibitem[Kob99]{kobal_k}
Damjan Kobal, \emph{{$K$}-theory, {H}ermitian {$K$}-theory and the {K}aroubi
  tower}, $K$-Theory \textbf{17} (1999), no.~2, 113--140. \MR{1696428}

\bibitem[Lev16]{lev_intro}
Marc Levine, \emph{An overview of motivic homotopy theory}, Acta Math. Vietnam.
  \textbf{41} (2016), no.~3, 379--407. \MR{3534540}

\bibitem[Lur09]{lurie_htt}
Jacob Lurie, \emph{Higher topos theory}, Annals of Mathematics Studies, vol.
  170, Princeton University Press, Princeton, NJ, 2009. \MR{2522659}

\bibitem[Lur11a]{DAGVII}
Jacob Lurie, \emph{Derived algebraic geometry {VII}: Spectral schemes},
  Available at \url{https://www.math.ias.edu/~lurie/papers/DAG-XI.pdf}.

\bibitem[Lur11b]{DAGXI}
\bysame, \emph{Derived algebraic geometry {XI}: Descent theorems}, Available at
  \url{https://www.math.ias.edu/~lurie/papers/DAG-XI.pdf}.

\bibitem[Lur17]{ha}
Jacob Lurie, \emph{Higher {A}lgebra}, 2017, Draft available from author's
  website as \url{https://www.math.ias.edu/~lurie/papers/HA.pdf}.

\bibitem[MNN17]{mnn_nilpotence}
Akhil Mathew, Niko Naumann, and Justin Noel, \emph{Nilpotence and descent in
  equivariant stable homotopy theory}, Adv. Math. \textbf{305} (2017),
  994--1084. \MR{3570153}

\bibitem[MS16]{mathew_stojanoska}
Akhil Mathew and Vesna Stojanoska, \emph{The {P}icard group of topological
  modular forms via descent theory}, Geom. Topol. \textbf{20} (2016), no.~6,
  3133--3217. \MR{3590352}

\bibitem[MV99]{morel_voev}
Fabien Morel and Vladimir Voevodsky, \emph{{${\bf A}^1$}-homotopy theory of
  schemes}, Inst. Hautes \'Etudes Sci. Publ. Math. (1999), no.~90, 45--143
  (2001). \MR{1813224}

\bibitem[NS{\O}09]{nso_landweber}
Niko Naumann, Markus Spitzweck, and Paul~Arne {\O}stv{\ae}r, \emph{Motivic
  {L}andweber exactness}, Doc. Math. \textbf{14} (2009), 551--593. \MR{2565902}

\bibitem[NS{\O}15]{nso_uniqueness}
\bysame, \emph{Existence and uniqueness of {$E_\infty$} structures on motivic
  {$K$}-theory spectra}, J. Homotopy Relat. Struct. \textbf{10} (2015), no.~3,
  333--346. \MR{3385689}

\bibitem[R{\O}16]{ro_slices}
Oliver R\"ondigs and Paul~Arne {\O}stv\ae{}r, \emph{Slices of hermitian
  {$K$}-theory and {M}ilnor's conjecture on quadratic forms}, Geom. Topol.
  \textbf{20} (2016), no.~2, 1157--1212. \MR{3493102}

\bibitem[Rob15]{robali_motivic}
Marco Robalo, \emph{{$K$}-theory and the bridge from motives to noncommutative
  motives}, Adv. Math. \textbf{269} (2015), 399--550. \MR{3281141}

\bibitem[Rog08]{rognes_galois}
John Rognes, \emph{Galois extensions of structured ring spectra. {S}tably
  dualizable groups}, Mem. Amer. Math. Soc. \textbf{192} (2008), no.~898,
  viii+137. \MR{2387923}

\bibitem[RS{\O}18]{rso_homlim}
Oliver R\"{o}ndigs, Markus Spitzweck, and Paul~Arne {\O}stv\ae{}r, \emph{The
  motivic {H}opf map solves the homotopy limit problem for {$K$}-theory}, Doc.
  Math. \textbf{23} (2018), 1405--1424. \MR{3874943}

\bibitem[S{\O}09]{so_bott}
Markus Spitzweck and Paul~Arne {\O}stv{\ae}r, \emph{The {B}ott inverted
  infinite projective space is homotopy algebraic {$K$}-theory}, Bull. Lond.
  Math. Soc. \textbf{41} (2009), no.~2, 281--292. \MR{2496504}

\bibitem[Sus84]{suslin}
Andrei~A. Suslin, \emph{On the {$K$}-theory of local fields}, Proceedings of
  the {L}uminy conference on algebraic {$K$}-theory ({L}uminy, 1983), vol.~34,
  1984, pp.~301--318. \MR{772065}

\bibitem[Voe98]{MR1648048}
Vladimir Voevodsky, \emph{{$\bold A^1$}-homotopy theory}, Proceedings of the
  {I}nternational {C}ongress of {M}athematicians, {V}ol. {I} ({B}erlin, 1998),
  no. Extra Vol. I, 1998, pp.~579--604. \MR{1648048}

\end{thebibliography}
\bibliographystyle{amsalpha}

\end{document}